\documentclass[a4paper,11pt]{amsart}
\usepackage{amsfonts,amsthm,amsmath,amssymb,graphicx}
\theoremstyle{plain}
\newtheorem{lem}{Lemma}

\newtheorem{thm}[lem]{Theorem}
\newtheorem{cor}[lem]{Corollary}
\theoremstyle{definition}

\theoremstyle{remark}

\numberwithin{equation}{section}

\renewcommand{\epsilon}{\varepsilon}

\usepackage[usenames]{color}
\usepackage{ulem}

\title[Distribution of fractional parts]{
Metrical results on the distribution of fractional parts of powers of real numbers  }
\date{}

\author{Yann Bugeaud}
\address{Yann Bugeaud\\IRMA UMR 7501, CNRS,
Universit\'e de Strasbourg, 7, rue Ren\'e Descartes, 67084 Strasbourg, France}
\email{bugeaud@math.unistra.fr}
\author{Lingmin Liao}
\address{Lingmin Liao\\LAMA UMR 8050, CNRS,
Universit\'e Paris-Est Cr\'eteil, 61 Avenue du
G\'en\'eral de Gaulle, 94010 Cr\'eteil Cedex, France}
\email{lingmin.liao@u-pec.fr}
\author{Micha\l\ Rams}
\address{Micha\l\ Rams\\Institute of Mathematics\\ Polish
Academy of Sciences\\ ul.
\'Sniadeckich 8, 00-656 Warszawa\\ Poland }
\email{M.Rams@impan.pl}
\thanks{M. R. was supported by National Science Centre grant
2014/13/B/ST1/01033 (Poland).}

\begin{document}
\begin{abstract}
Denote by $\{\cdot\}$ the fractional part. 
We establish several new metrical results on the distribution properties of the sequence 
$(\{x^n\})_{n\geq 1}$. 
Many of them are presented in a more general framework, in which the 
sequence of functions $(x \mapsto x^n)_{n \ge 1}$ is replaced by a sequence 
$(f_n)_{n \ge 1}$, under some growth and regularity conditions on the functions $f_n$. 
%
\end{abstract}

\def\thefootnote{}
\def\thefootnote{\arabic{footnote}}

\maketitle

\section{Introduction}
Let $\{ \cdot \}$ denote the fractional part and 
$\| \cdot \|$ the distance to the nearest integer. 
For a given real number $x > 1$,
only few results are known on the distribution of the sequence $(\{ x^n \})_{n \ge 1}$.
For example, we still do not know whether $0$ is a limit point of 
$(\{ {\rm e}^n \})_{n \ge 1}$, nor of $(\{ ({\frac{3}{2}})^n \})_{n \ge 1}$; 
see \cite{Bu12} for a survey of related results. 

However, several metric statements have been established.
The first one was obtained
in 1935 by Koksma \cite{Kok}, who proved that for almost every $x>1$ 
the sequence $(\{x^n\})_{n\geq 1}$ is uniformly distributed on the unit interval $[0,1]$. 
Here and below, almost every always refers to the Lebesgue measure. 
In 1967, Mahler and Szekeres \cite{MS} studied the quantity 
\[
P(x):=\liminf \|x^n\|^{1/n} \quad (x>1). 
\]
They proved that $P(x)=0$ if $x$ is transcendental and $P(x)=1$ for almost all $x>1$. 
The function $x \mapsto P(x)$ was 
subsequently studied in 2008 by Bugeaud and Dubickas \cite{BD}. 
Among other results, 
it was shown in \cite{BD} that, for all $v>u>1$ and $b > 1$, we have 
\[
\dim_H \{x\in (u, v): P(x) \leq 1/b \} = \frac{\log v}{\log (bv)},
\] 
where $\dim_H$ denotes the Hausdorff dimension. 

In a different direction, Pollington \cite{Po} showed in 1980 
that there are many real numbers $x > 1$ such that $(\{x^n\})_{n\geq 1}$
is very far from being well distributed, namely he established that, for any $\epsilon>0$, we have 
\[
\dim_H \big\{x>1 : \{x^n\} < \epsilon \ \text{for all} \ n \big\} = 1.
\]
This result has been subsequently extended by Bugeaud and Moshche\-vitin \cite{BM} 
and, independently, by Kahane \cite{Kah}, who proved that for any $\epsilon>0$, 
for any sequence of real numbers $(y_n)_{n\geq 1}$, we have
\[
\dim_H \big\{x>1 : \|x^n-y_n\| < \epsilon \ \text{for all} \ n \big\} = 1. 
\]


In the present paper, we further investigate, from a metric point of view, 
the Diophantine approximation 
properties of the sequence $(\{x^n\})_{n\geq 1}$, where $x>1$, and extend several 
known results to more general families of 
sequences $(\{ f_n(x) \})_{n\geq 1}$, under some conditions on
the sequence of functions $(f_n)_{n \ge 1}$.


As a consequence of our main theorem, we obtain an inhomogeneous version of the result 
of Bugeaud and Dubickas \cite{BD} mentioned above. 

\begin{thm}\label{main}
Let $b>1$ be a real number 
and $y=(y_n)_{n \geq 1}$ an arbitrary sequence of real numbers in $[0,1]$. Set 
\[
E(b, y):= \{x >1: \|x^n -y_n\| < b^{-n} \text{ for
infinitely many } n\}.
\]
For every $v > 1$, we have    
\[
\lim_{\epsilon\to 0} {\dim_H ([v-\epsilon, v+\epsilon]\cap E(b, y))}=\frac{\log v}{\log (bv)}.
\]
\end{thm}

In the homogeneous case (that is, the case where $y_n=0$ for $n \ge 1$), 
Theorem \ref{main} was proved in \cite{BD} 
by using a classical result of Koksma \cite{Kok45} and the mass transference 
principle developed by Beresnevich and Velani \cite{BV}. 
The method of \cite{BD} still works 
when $y$ is a constant sequence, but one then needs 
to apply the inhomogeneous version of Koksma's theorem in \cite{Kok45}.  
Here, for an arbitrary sequence $(y_n)_{n\geq 1}$, we use a direct construction. 

Letting $v$ tend to infinity in Theorem \ref{main}, 
we obtain the following immediate corollary.

\begin{cor}
For an arbitrary sequence $y$ of real numbers in $[0,1]$ and any real number $b > 1$, 
the set $E(b, y)$ has full Hausdorff dimension.  
\end{cor}

Theorem \ref{main} gives, for every $v > 1$, the value 
of the localized Hausdorff dimension of $E(b, y)$ at the point $v$. 
We stress that, in the present context, the localized Hausdorff dimension varies 
with $v$, while this is not at all the case for many classical results, including 
the Jarn\'\i k--Besicovitch Theorem and its extensions. 
Taking this point of view allows us also to place
Theorem \ref{main} in a more general context, where the family of functions 
$x \mapsto x^n$ is replaced by an arbitrary family of functions $f_n$ satisfying some
regularity and growth conditions. 

We consider a family of strictly positive increasing $C^1$ 
functions $f=(f_n)_{n\geq 1}$ defined on an open interval $I \subset \mathbb{R}$ and
such that $f_n(x), f_n'(x)> 1$ for all $x\in I$. For $\tau>1$, define
\[
E(f, y, \tau):= \{x\in I: \|f_n(x) -y_n\| < f_n(x)^{-\tau} \text{ for
infinitely many } n\}.
\]
For $v\in I$, put
\[
u(v):=\limsup_{n\to\infty} {\log f_{n}(v) \over \log f_n'(v)}, 
\quad \ell(v):=\liminf_{n\to\infty} {\log f_{n}(v) \over \log f_n'(v)}.
\]

We will assume the regularity condition 
\begin{equation}\label{cond-ext}
\lim_{r\to 0} \limsup_{n\to\infty} \sup_{|x-y|<r} {\log f_n'(x) \over \log f_n'(y)}=1,
\end{equation}
which guarantees the continuity of the functions $u$ and $\ell$. 

For non-linear functions $f_n$, i.e., when $f_n$ is not of the form 
$f_n(x)=a_n\cdot x+b_n$, we also need the following condition:
\begin{equation}\label{cond-ext-2}
M:=\sup_{n\geq 1} {\log f_{n+1}'(v) \over \log f_n'(v)}<\infty \quad \text{ for all} \ v\in I.
\end{equation}

Theorem \ref{main} is a particular case of the following general statement. 

\begin{thm}\label{asym-general} 
Consider a family of strictly positive increasing $C^1$ 
functions $f=(f_n)_{n\geq 1}$ defined on an open interval $I \subset \mathbb{R}$ and
such that $f_n(x), f_n'(x)> 1$ for all $x\in I$.
Assume \eqref{cond-ext} and \eqref{cond-ext-2}.
If for all $x\in I$, 
\begin{align}\label{general-cond-3} 
\forall \epsilon>0, \quad  \sum_{n=1}^\infty f_n'(x)^{-\epsilon} < \infty,
\end{align}
 then, for any $v\in I$ and any $\tau > 1$, we have
\[  
{1 \over 1+\tau u(v)} \leq \lim_{\epsilon\to 0} 
{\dim_H ([v-\epsilon, v+\epsilon]\cap E(f, y, \tau))} \leq  {1 \over 1+\tau \ell(v)}.
\]

If the functions $f_n$ are linear then 
we do not need to assume \eqref{cond-ext-2}, and the assertion gets strengthened to

\[
\lim_{\epsilon\to 0} \dim_H ([v-\epsilon, v+\epsilon]\cap E(f,y,\tau)) = \frac 1 {1+\tau \ell(v)}.
\]
\end{thm}
 
We remark that the condition (\ref{general-cond-3}) is satisfied if 
\begin{align}\label{general-cond-2} 
\forall x\in I, \quad \lim_{n\to\infty} {\log f'_{n}(x) \over \log n} =\infty.
\end{align}
We also observe that the condition (\ref{cond-ext}) implies that $\ell(v)\geq 1$ for $v$ in $I$. 
In many cases (in particular, for $f_n(x)=x^n$), we have $u(v)=\ell(v)=1$
for $v$ in $I$. 

It follows from the formulation of Theorem \ref{asym-general}
that the real number $\tau$ can be replaced by a 
continuous function $\tau: I \rightarrow (0,\infty)$, in which case 
the set $E(f, y, \tau)$ is defined by 
\[
E(f, y, \tau):= \{x\in I: \|f_n(x) -y_n\| < f_n(x)^{-\tau(x)} \text{ for
infinitely many } n\}.
\] 
We get at once the following localized version 
of Theorem \ref{asym-general}. 
For the classical Jarn\'\i k--Besicovitch Theorem, such a localized theorem 
was obtained by Barral and Seuret \cite{BS}, who were the first to 
consider localized Diophantine approximation. 

\begin{cor}\label{Cor:main}
With the above notation and 
under the hypotheses of Theorem \ref{asym-general}, we have
\[
\frac 1 {1+\tau(v) u(v)} \leq \lim_{\epsilon\to 0} 
\dim_H ([v-\epsilon, v+\epsilon]\cap E(f,y,\tau)) \leq \frac 1 {1+\tau(v) \ell(v)}.
\]
\end{cor}

 We illustrate Theorem \ref{asym-general} and 
Corollary \ref{Cor:main} by some examples. 
If the family of functions $f = (f_n)_{n \ge 1}$ 
in Theorem \ref{asym-general} is such that, for every $x$ in $I$, the sequence
$(f_n(x))_{n \ge 1}$ increases sufficiently rapidly, then 
\[
\lim_{\epsilon\to 0} \dim_H ([v-\epsilon, v+\epsilon]\cap E(f,y,\tau)) = \frac 1 {1+\tau},
\]
independently of the family $f$. This applies, for example, to the families of functions 
$x^{n^2}, x^n, 2^nx$ and $x^{\sqrt{n}}$. 

The case $f_n(x)=a_nx$, where $(a_n)_{n \ge 1}$ is an increasing sequence 
of positive integers, has been studied by Borosh and Fraenkel \cite{BF}
(but only in the special case of a constant sequence $y$ equal to $0$).  
Let $I$ be an open, non-empty, real interval. They proved that 
$$
\dim_H  \{ x \in I : \| a_n x \| < a_n^{-\tau} \} = \frac{1 + s}{1 + \tau},
$$
where $s$
(usually called the convergence exponent of the sequence $(a_n)_{n \ge 1}$)
is the largest real number in $[0, 1]$ such that
$$
\sum_{n \ge 1} a_n^{-s-\epsilon} \quad \hbox{converges for any $\epsilon > 0$}. 
$$
The case $s=0$ of their result, which corresponds 
to rapidly growing sequences $(a_n)_{n \ge 1}$, 
follows from Theorem \ref{asym-general}. The case $a_n = n$ 
for $n \ge 1$ corresponds to the Jarn\'\i k--Besicovitch Theorem. 
We stress that the assumption (\ref{general-cond-3})
is satisfied only if $(a_n)_{n \ge 1}$ increases
sufficiently rapidly.

\bigskip

Questions of uniform Diophantine approximation were recently studied 
by Bugeaud and Liao \cite{BuLi} for the $b$-ary and $\beta$-expansions 
and by Kim and Liao \cite{KL} for the irrational rotations. 
In this paper, we consider the uniform Diophantine approximation 
of the sequence $(\{x^n\})_{n\geq 1}$ with $x>1$.

For a real number $B > 1$ and 
a sequence of real numbers $y=(y_n)_{n\geq 1}$ in $[0,1]$, set 
\begin{align*}
F(B,y):=\{x>1: \ \text{for all large integer} \ N, \  \|x^n-y_n\| < B^{-N} \\
\text{ has a solution } 1\leq n \leq N\}.
\end{align*}
Our next theorem gives a lower bound for the 
Hausdorff dimension of $F(B,y)$ intersected with a small interval.

\begin{thm}\label{thm-unif}
Let $B>1$ be a real number 
and $y$ an arbitrary sequence of real numbers in $[0,1]$. 
For any $v > 1$, we have    
$$
\lim_{\epsilon \to 0} \dim_H([v-\epsilon, v+\epsilon]\cap F(B,y))
\geq \left({\log v-\log B \over \log v+\log B}\right)^2.
$$
\end{thm}

Unfortunately, we are unable to decide whether the inequality in 
Theorem \ref{thm-unif} is an equality. Observe that the lower bound we obtain is the same
as the one established in \cite{BuLi} for a question of uniform Diophantine approximation 
related to $b$-ary and $\beta$-expansions. 

Letting $v$ tend to infinity, we have the following corollary. 
\begin{cor}
For an arbitrary sequence $y$ of real numbers in $[0,1]$ and any real number $B > 1$, 
the set $F(B, y)$ has full Hausdorff dimension. 
\end{cor}

\medskip

We end this paper with results on sequences $(\{x^n\})_{n\geq 1}$, with $x>1$, which are 
badly distributed, in the sense that all of their points lie in a small interval. 
As above, we take a more general point of view. 
Consider a family of $C^1$ strictly positive increasing functions $f=(f_n)_{n\geq 1}$ 
defined on an open interval $I \subset \mathbb{R}$ 
and such that $f_n(x), f_n'(x)> 1$ for all $x\in I$ 
and for all $n\geq 1$. Let $\delta=(\delta_n)_{n\geq 1}$ 
be a sequence of positive real numbers such that $\delta_n< 1/4$ for $n \ge 1$. Set 
\[
G(f, y, \delta):=\{x\in I:  \|f_n(x)-y_n\|\leq \delta_n, \ \forall n\geq 1  \}.
\]
We need the following hypotheses:
\begin{align}\label{hyp-1}
\forall \epsilon>0, \forall n\geq 1, \ {\inf_{x\in (v-\epsilon,v+\epsilon)} f_{n+1}'(x) 
\over \sup_{x\in (v-\epsilon,v+\epsilon)} f_n'(x)} \cdot \delta_n \geq 2, 
\end{align}
\begin{align}\label{hyp-2}
\forall x\in I, \quad\lim_{n\to\infty} {\log f_{n+1}'(x) \over \log f_n'(x)}=\infty.
\end{align}

Our last main theorem is as follows.

\begin{thm}\label{bad-general}
Keep the above notation. 
Under the hypotheses (\ref{cond-ext}), (\ref{hyp-1}), 
and (\ref{hyp-2}), for all $v\in I$, we have  
\begin{equation}\label{formula:thm-bad}
\lim_{\epsilon\to0}\dim_H ([v-\epsilon, v+\epsilon]\cap G(f, y, \delta)) 
= \liminf\limits_{n\to \infty} { \log f_n'(v)+ {\sum\limits_{j=1}^{n-1} \log\delta_j}  
\over \log f_n'(v)-  {\log \delta_n}}.
\end{equation}
\end{thm}

We remark that our result extends a recent result of Baker \cite{Ba}. 
In fact, in \cite{Ba}, the author studied the special case $f_n(x)=x^{q_n}$ 
with $(q_n)_{n\geq 1}$ being a strictly increasing sequence of real numbers such that 
\[
 \lim_{n\to\infty} (q_{n+1}-q_n) = + \infty.
\]
Our result also gives the following corollary.

\begin{cor}
Let $(a_n)_{n\geq 1}$ be a sequence of positive real numbers such that 
\[ 
\lim_{n\to\infty } {a_{n+1} \over a_n} = + \infty.
\]
Then, for any sequence $(y_n)_{n\geq 1}$ of real numbers, we have 
\[
\dim_H \{x\in \mathbb{R}: \lim_{n\to + \infty}\|a_nx-y_n\|=0\}=1.
\]
\end{cor}

%
%
\bigskip
\section{Basic tools}

\smallskip
We present two lemmas which serve as important tools for estimating 
the Hausdorff dimension of the sets studied in this paper.

Let $[0,1]=E_0\supset E_1\supset E_2\supset \cdots$ be a decreasing sequence of sets, 
with each $E_k$ a finite union of disjoint closed intervals. 
The components of $E_k$ are called $k$-th level basic intervals. 
Set $F=\cap_{k=0}^\infty E_k$. 
We do not assume that each basic interval in $E_{k-1}$ contains the same number of next level 
basic intervals, nor that they are of the same length, nor that the gaps between two consecutive 
basic intervals are equal. 
Instead, for $x\in E_{k-1}$, we denote by $m_k(x)$ the number of $k$-th level 
basic intervals contained in the $(k-1)$-th level basic interval containing $x$, 
and by $\tilde{\varepsilon}_k(x)$ the minimal distance between two of them. Set
\[
\varepsilon_k(x) = \min_{i\leq k} \tilde{\varepsilon}_i(x).
\]
In the following, we generalize a lemma in Falconer's book \cite[Example 4.6]{Fa1}.

\begin{lem}\label{local}
For any open interval $I\subset [0,1]$ intersecting $F$, we have
\[
\dim_H(I\cap F)\geq \inf_{x\in I\cap F}\liminf_{k\to\infty} 
\frac{\log (m_1(x)\cdots m_{k-1}(x))}{-\log (m_k(x)\epsilon_k(x))}.
\]
\end{lem}
\begin{proof}
The proof is similar to that in the book of Falconer. 
We define a probability measure $\mu$ on $F$ by assigning the mass evenly. 
Precisely, for $k\geq 1$, let $I_k(x)$ be the $k$-th level interval containing $x$.  
For $x\in F$ and $k\geq 1$, we put a mass $(m_1(x)\cdots m_{k}(x))^{-1}$ 
to the interval $I_k(x)$. Note that any two $k$-th basic intervals contained in the same $(k-1)$-th interval have the same measure. One can check that the measure $\mu$ is well defined.

Now let us calculate the local dimension at the point $x$.  
Let $B(x,r)$ be the ball of radius $r$ centered at $x$. Suppose that 
 $\epsilon_k(x) \leq 2r < \epsilon_{k-1}(x)$. 
 The number of $k$-th level intervals intersecting $B(x,r)$ is at most 
 \[
 \min\left\{m_k(x), \ {2r \over \epsilon_k(x)}+1\right\}  \leq \min
 \left\{m_k(x), \ {4r \over \epsilon_k(x)}\right\}\leq m_k(x)^{1-s}\left({4r \over \epsilon_k(x)}\right)^s,
 \]
 for any $s\in[0,1]$.
 Thus 
 \[
 \mu(B(x,r)) \leq m_k(x)^{1-s}\left({4r \over \epsilon_k(x)}\right)^s \cdot (m_1(x)\cdots m_{k}(x))^{-1}.
 \]
 Hence
 \begin{eqnarray*}
 {\log \mu(B(x,r)) \over \log r} 
 \geq  { s \log m_k(x)\epsilon_k(x) - s\log (4r) + \log (m_1(x)\cdots m_{k-1}(x)) \over -\log r}.
 \end{eqnarray*}
Let $s$ be in $(0, 1)$ such that  
\[
s< \inf_{z\in I\cap F} 
\liminf_{k\to\infty} \frac{\log (m_1(z)\cdots m_{k-1}(z))}{-\log m_k(z)\epsilon_k(z)}
\leq \liminf_{k\to\infty} \frac{\log (m_1(x)\cdots m_{k-1}(x))}{-\log m_k(x)\epsilon_k(x)}.
\]
Then
\[
{s \log m_k(x)\epsilon_k(x) - s\log 4 + \log (m_1(x)\cdots m_{k-1}(x)) } \geq 0,
\]
for $k$ large enough. Therefore 
\[
\liminf_{r\to0}{\log \mu(B(x,r)) \over \log r} \geq s.
\]
The proof is completed by applying the mass distribution principle (see \cite{Fa2}, Proposition 2.3).
\end{proof}

We also have an upper bound for the dimension of the set $I\cap F$. 
Denote by $|I_k(x)|$ the length of the $k$-th basic interval $I_k(x)$ containing $x$.
\begin{lem}\label{local-upper}
For any open interval $I\subset [0,1]$ intersecting $F$, we have
\[
\dim_H(I\cap F)\leq \sup_{x\in I\cap F}\liminf_{k\to\infty} 
\frac{\log (m_1(x)\cdots m_{k}(x))}{-\log |I_k(x)|}.
\]
\end{lem}
\begin{proof}
We define the same probability measure $\mu$ as in Lemma \ref{local}, i.e.,  
the interval $I_k(x)$ has measure $(m_1(x)\cdots m_{k}(x))^{-1}$. 
Then 
\[
\liminf_{r\to0}{\log \mu(B(x,r)) \over \log r} \leq 
\liminf_{k\to\infty} \frac{\log \mu(I_k(x))}{-\log |I_k(x)|}
=\liminf_{k\to\infty} \frac{\log (m_1(x)\cdots m_{k}(x))}{-\log |I_k(x)|}.
\]
We finish the proof by applying again the mass distribution principle (see \cite{Fa2}, Proposition 2.3).
\end{proof}

\bigskip
\section{Asymptotic approximation}\label{sec-asy}

In this section, we prove Theorem \ref{asym-general}.
To see that Theorem \ref{main} is a special case of it, take
the family of functions $f$ defined by 
\[
f_n(x)=x^n, \quad \forall n\geq 1,
\]
we have $u(v)=\ell(v)=1$ and
\begin{eqnarray*}
[v-\epsilon, v+\epsilon]\cap E\left(f, y, {\log b \over \log (v+\epsilon)}\right) \subset [v-\epsilon, v+\epsilon]\cap E(b,y) \\
\subset [v-\epsilon, v+\epsilon]\cap E\left(f, y, {\log b \over \log (v-\epsilon)}\right).
\end{eqnarray*}
Then, Theorem \ref{main} follows directly from Theorem \ref{asym-general}.

Now we prove Theorem \ref{asym-general}.

\begin{proof}[Proof of Theorem \ref{asym-general}]

{\bf Lower bound}: 
We can assume that $u(v)$ is finite, since otherwise there is nothing to prove.
Let us start by the simple observation about the condition \eqref{cond-ext}. Given
an integer $n \ge 1$, set
\begin{equation}\label{def:eta}
\eta(n) = \sup\left\{\frac {\log f_n'(w)}{\log f_n'(z)}-1; 
w,z \in [v-\epsilon, v+\epsilon], \lvert f_n(w)-f_n(z)\rvert \leq 1\right\}.
\end{equation}

\begin{lem} \label{lem:local}
If \eqref{cond-ext} and \eqref{general-cond-3} hold, then
\[
\lim_{n\to \infty} \eta(n) = 0.
\]
\end{lem}
\begin{proof}
Assume this is not true. Then there exists a sequence of integers $(n_i)$ and 
a sequence of pairs of points $(w_i, z_i)$ 
such that $\lvert f_{n_i}(w_i) - f_{n_i}(z_i)\rvert \leq 1$ and 
$$
{\log f_{n_i}'(w_i) \over \log f_{n_i}'(z_i)} > Z > 1.
$$ 
By compactness of $[v-\epsilon, v+\epsilon]$, taking a subsequence if necessary,
we can assume that $(w_i)_{i \ge 1}$ converges to some point $w_0$.

By \eqref{general-cond-3}, $f_n'(v)\to\infty$. Hence, \eqref{cond-ext} gives us

\[
\lim_{n\to\infty} \inf_{x\in [v-\epsilon, v+\epsilon]} f_n'(x) = \infty.
\]
This implies that 
$$
\lvert w_i-z_i\rvert \leq {1 \over \inf_{x\in [v-\epsilon, v+\epsilon]} f_{n_i}'(x) } \to 0 
$$ 
as $i\to\infty,$ and hence any neighborhood of $w_0$ 
contains all except finitely many points $w_i, z_i$. 
Thus, in any neighbourhood $U$ of $w_0$ we have 
\[
\limsup_{n\to\infty} \sup_{w,z\in U} \frac {\log f_n'(w)}{\log f_n'(z)} > Z,
\]
which is a contradiction with \eqref{cond-ext}.
\end{proof}


Now we construct a nested Cantor set which is the intersection 
of unions of subintervals at level $n_i$, where $(n_i)_{i \ge 1}$ is an increasing 
sequence of positive integers which will be defined precisely later.  
Suppose we have already well chosen this subsequence. 
Let us describe the nested family of subintervals.
For each level $i$, we need to consider the set of points $x$ such that
\[
\|f_{n_i}(x)-y_{n_i}\|\leq f_{n_i}(x)^{-\tau}.
\]
By the property $\|f_{n_1}(x)-y_{n_1}\|\leq f_{n_1}(x)^{-\tau}$, we take the intervals at level $1$ as 
\[
I_1(k, v, f, y, \tau):=[f_{n_1}^{-1}(k+y_{n_1}-f_{n_1}(v+\epsilon)^{-\tau}), 
f_{n_1}^{-1}(k+y_{n_1}+f_{n_1}(v+\epsilon)^{-\tau})],
\]
with $k$ being an integer in $[f_{n_1}(v-\epsilon)+1, \quad  f_{n_1}(v+\epsilon)-1]$. 

Suppose we have constructed the intervals at level $i-1$. 
Let $[c_{i-1}, d_{i-1}]$ be an interval at such level.
A subinterval of $[c_{i-1}, d_{i-1}]$ at level $i$ is such that 
\[
[f_{n_i}^{-1}(k+y_{n_i}-f_{n_i}(d_{i-1})^{-\tau}), f_{n_i}^{-1}(k+y_{n_i}+f_{n_i}(d_{i-1})^{-\tau})],
\]
with $k$ being an integer in $[f_{n_i}(c_{i-1})+1, \ f_{n_i}(d_{i-1})-1]$.
By continuing this construction, we obtain intervals $I_i(\cdot)$ for all levels.
%

Finally, the intersection $F$ of these nested intervals is obviously a subset 
of $[v-\epsilon, v+\epsilon] \cap E(f, y, \tau)$.

Let $z\in F$ and $[c_{i}(z), d_{i}(z)]$ be the $i$-th level interval containing $z$. Then we have
\begin{equation}\label{estimate-mi}
m_{i+1}(z) \geq f_{n_{i+1}}'(w_{i}) \cdot (d_{i}-c_{i})-2\geq 
f_{n_{i+1}}'(w_{i}) \cdot {2f_{n_i}(d_{i})^{-\tau} \over f_{n_i}'(z_i)}-2,
\end{equation}
where $w_i, z_i\in [c_{i}(z), d_{i}(z)]$.
Furthermore,
\begin{equation}\label{estimate-epsiloni}
\epsilon_{i+1}(z) \geq  {1-2 f_{n_{i+1}}(c_{i}(z))^{-\tau} \over f_{n_{i+1}}'(u_i)} \geq {1 \over 2f_{n_{i+1}}'(u_i)},
\end{equation}
where $u_i\in [c_{i}(z), d_{i}(z)]$.

Now we are going to define the subsequence $(n_i)_{i \ge 1}$.

\begin{lem}\label{lem:distortion}
 Assume (\ref{cond-ext}) and (\ref{cond-ext-2}). 
For any $\gamma>0$, we can find a subsequence $(n_i)_{i \ge 1}$ such that
\begin{equation}\label{seq-ni}
{f_{n_{i+1}}'(w) \over f_{n_{i+1}}'( u)} \leq f_{n_{i}}'( z)^{\gamma} \quad \forall w, u\in [c_{i}( z), d_{i}(z)], 
\end{equation}
{and for any small $\epsilon>0$, we have
\begin{align}\label{cond-3}
 \forall x\in (v-\epsilon,v+\epsilon), \quad \lim_{i\to\infty} {\log f_{n_i}'(x) \over \log f'_{n_{i-1}}(x)}=\lim_{i\to\infty} {\log f_{n_i}'(x) \over \log f_{n_{i-1}}(x)}=\infty,
\end{align}}
and
\begin{align}\label{cond-1}
 {\inf_{x\in (v-\epsilon,v+\epsilon)} f_{n_{i+1}}'(x) 
 \over \sup_{x\in (v-\epsilon,v+\epsilon)} f_{n_i}'(x) \cdot f_{n_i}(x)^\tau } \geq 2. 
\end{align}
If $f_n$ are linear then we do not need to assume \eqref{cond-ext-2}, 
moreover we can choose $(n_i)$ in such a way that we have 
(in addition to the other parts of the assertion)
\begin{equation} \label{eqn:linear}
\lim_{i\to\infty} \frac {\log f_{n_i}(v)}{\log f_{n_i}'(v)} = \ell(v).
\end{equation}
\end{lem}

\begin{proof}
In the linear case \eqref{seq-ni} is automatically true, 
and to have \eqref{cond-3} and \eqref{cond-1} we just need that $(n_i)_{i \ge 1}$ increases 
sufficiently fast (as will be clear from the proof for the general case). 
Hence, we will be free to choose $(n_i)$ satisfying in addition \eqref{eqn:linear}.

Let us proceed with the general case. For any $\gamma>0$, by Lemma \ref{lem:local},
there exists $n_0\in \mathbb{N}$ such that 
\[
\forall n\geq n_0, \quad \eta(n)< {\gamma \over 2M},
\]
where $M$ is the constant in assumption (\ref{cond-ext-2}).


Starting with this $n_0$, by the assumption (\ref{cond-ext-2}), 
we can then construct a subsequence $(n_i)_{i \ge 1}$ satisfying
\begin{equation}\label{Ki}
{\gamma \over 2\eta(n_i) \cdot M} \leq 
{\log f_{n_{i+1}}'(v) \over \log f_{n_{i}}'(v)} \leq{\gamma \over 2\eta(n_i)}.
\end{equation}

Observe that, as $\eta(n_i)\to 0$ by Lemma \ref{lem:local}, 
the lefthand side of \eqref{Ki} implies the first part of  \eqref{cond-3}. 
As $u<\infty$, the second part of \eqref{cond-3} follows. 
The condition \eqref{cond-1} will also follow, provided that $n_0$ was selected large enough.

We need now to prove \eqref{seq-ni}. 
{By (\ref{def:eta}), for any $w, u$ in the interval $[c_i(z),d_i(z)]$,
\begin{equation}\label{wz}
{f_{n_{i+1}}'(w) \over f_{n_{i+1}}'(u) } 
\leq {f_{n_{i+1}}'(z)^{1+\eta(n_i)} \over f_{n_{i+1}}'(z)^{1-\eta(n_i)}} = f_{n_{i+1}}'(z)^{2\eta(n_i)}.
\end{equation}}
Combining (\ref{Ki}) and (\ref{wz}), we get (\ref{seq-ni}).

\end{proof}

We continue the proof of the lower bound of Theorem \ref{asym-general}.
By (\ref{estimate-mi}) and (\ref{cond-1}),
\[
m_{i+1}(z)  \geq  f_{n_{i+1}}'(w_{i}) \cdot {2f_{n_i}(d_{i})^{-\tau} \over f_{n_i}'(z_i)}-2 \geq 2,
\]
which then implies that $F$ is non-empty. 
%
Further, by (\ref{seq-ni}), for any $\gamma>0$, 
\begin{align}\label{m-largerthan}
m_{i+1}(z) \geq f_{n_{i+1}}'(z) \cdot f_{n_i}'(z)^{-\gamma}\cdot { {f_{n_i}(d_i)}^{-\tau} \over f_{n_i}'(z_i)}.
\end{align}
By (\ref{estimate-mi}), (\ref{estimate-epsiloni}) and (\ref{seq-ni}),
\begin{align}\label{me-largerthan}
m_{i+1}(z)\epsilon_{i+1}(z)  \geq f_{n_{i}}'(z)^{-\gamma} \cdot {f_{n_i}(d_i)^{-\tau} \over 2f_{n_i}'(z_i)}.
\end{align}
Thus, (\ref{general-cond-3}) and (\ref{cond-3}) 
imply that $-\log m_{i+1}(z)\epsilon_{i+1}(z)$ is unbounded. 
So by {(\ref{m-largerthan}), (\ref{me-largerthan})} and (\ref{cond-3}) 
\begin{align*}
& \liminf_{i\to\infty}\frac{\log (m_{2}(z)\cdots m_{i}(z))}{-\log m_{i+1}(z)\epsilon_{i+1}(z)}\\
 \geq &\liminf_{i\to\infty} 
 {\sum_{j={2}}^i (\log f_{n_{j}}'(z) - \gamma \log f_{n_{j-1}}'(z) - \tau \log f_{n_{j-1}}(d_j) 
 - \log f'_{n_{j-1}}(z_j)) \over \log 2+\log f_{n_{i}}'(z_i)+ \gamma \log f_{n_{i}}'(z)+\tau \log f_{n_{i}}(d_i)} \\
=& \liminf_{i\to\infty} \frac{\log f_{n_{i}}'(z)}{\log f_{n_{i}}'(z_i)+ \gamma \log f_{n_{i}}'(z)+\tau \log f_{n_{i}}(d_i)  }.
\end{align*}
Hence, by the definition of $\eta(n_i)$, we have
\begin{align*}
&\liminf_{i\to\infty}\frac{\log (m_{2}(z)\cdots m_{i}(z))}{-\log m_{i+1}(z)\epsilon_{i+1}(z)}\\ 
\geq  & {1  \over \limsup\limits_{i\to\infty} \left(1 +\eta(n_i) + \gamma+ \tau(1+\eta(n_i))
 \cdot   {\log f_{n_i}(d_i) \over \log f_{n_{i}}'(d_i) }\right)}.
\end{align*}
In the linear case, $\log f_{n_i}(d_i)/ \log f_{n_i}'(d_i)$
 converges to $\ell(\lim\limits_{i\to\infty}d_i)$. In the general situation, we have 
\[
\limsup\limits_{i\to\infty} {\log f_{n_i}(d_i) \over \log f_{n_{i}}'(d_i) } \leq u(\lim_{i\to\infty} d_i).
\]
 As $\gamma$ can be chosen arbitrarily small, $\eta(n_i) \to 0$ 
 by Lemma \ref{lem:local}, and 
 $$
 \lim\limits_{i\to\infty} d_i \in [v-\epsilon, v+\epsilon],
 $$  
 the lower bound is obtained by applying Lemma \ref{local}.

\medskip
{\bf Upper bound}: Since for all $x\in [v-\epsilon, v+\epsilon] \cap E(f, y, \tau)$, we have 
\[ \|f_n(x) -y_n\| < f_n(x)^{-\tau} \]
for infinitely many $n\geq 1$.  Then the set $[v-\epsilon, v+\epsilon] \cap E(f, y, \tau)$ 
is covered by the union of the family of intervals 
\[
I_n(k):=[f_{n}^{-1}(k+y_{n}-f_{n}(v-\epsilon)^{-\tau}), f_{n}^{-1}(k+y_{n}+f_{n}(v-\epsilon)^{-\tau})],
\]
where $k\in [f_{n}(v-\epsilon), \  f_{n}(v+\epsilon)]$ is an integer.
Note that the length of the interval $I_n(k)$ satisfies
\[
{ |I_n(k)| \leq {2f_n(v-\epsilon)^{-\tau} \over f_n'(z)} \quad \text{for some}\ z\in (v-\epsilon, v+\epsilon).}
\]
The number of the intervals at level $n$ is less than 
\[
{ f_{n}(v+\epsilon)-f_{n}(v-\epsilon)\leq 2\epsilon f'_n(w) \quad \text{for some}\ w\in (v-\epsilon, v+\epsilon).}
\]
Thus for $s>0$
\begin{align}\label{upper-sum}
\begin{split}
&\sum_{n=1}^\infty \sum_{k\in [f_{n}(v-\epsilon),  f_{n}(v+\epsilon)]} |I_n(k)|^s 
\leq  \sum_{n=1}^\infty  2\epsilon f'_n(w) \cdot \left({2f_n(v-\epsilon)^{-\tau} \over f_n'(z)}\right)^s.\\
\end{split}
\end{align}
By the definition of $\ell(v)$, for any $\eta>0$, there exists $n_0=n_0(\eta)\in \mathbb{N}$ 
such that for any $n \geq n_0$
\[
f_n(v-\epsilon)> f_n'(v-\epsilon)^{\ell(v-\epsilon)-\eta}.
\]
Thus by ignoring the first $n_0$ terms, we have (\ref{upper-sum}) is bounded by
\begin{align}\label{upper-sum2}
\begin{split}
  &2^{1+s}\epsilon \sum_{n=n_0}^\infty   f'_n(w) \cdot f'_n(z)^{-s} 
  \cdot \left(f_n'(v-\epsilon)\right)^{-\tau s (\ell(v-\epsilon)-\eta)}.
\end{split}
\end{align}

Hence by the assumption (\ref{general-cond-3}) if 
\begin{align*}
s> \limsup_{n\to\infty} {\log f_n'(w) \over \log f_n'(z)+ \tau (\ell(v-\epsilon)-\eta)\log f_n'(v-\epsilon)} 
\end{align*}
the sum in (\ref{upper-sum}) converges. By \eqref{cond-ext}, \[ \lim_{n\to\infty} 
{\log f_n'(w) \over \log f_n'(z)}=1, \quad \lim_{n\to\infty} {\log f_n'(w) \over \log f_n'(v-\epsilon)}= 1.\] 
Therefore 
\begin{align*}
\lim_{\epsilon \to 0} \dim_H  [v-\epsilon, v+\epsilon] \cap E(f, y, \tau) \leq  {1 \over 1 + \tau \ell(v) }.
\end{align*}
\end{proof}

\bigskip
\section{Uniform Diophantine approximation}

In this section, we study the uniform Diophantine 
approximation of the sequence $(\{x^n\})_{n\geq 1}$ with $x>1$.

Recall that for any sequence of real numbers $y=(y_n)_{n\geq 1}$ in $[0,1]$, 
we are interested in the set
\begin{align*}
F(B,y):=\{x>1: \ \text{for all large integer} \ N, \  \|x^n-y_n\| < B^{-N} \\
\text{ has a solution } 1\leq n \leq N\}.
\end{align*}
 

For any $v\in F(B,y)$, for any $\epsilon>0$, we will give a lower bound 
for the Hausdorff dimension of $[v-\epsilon, v+\epsilon]\cap F(B,y)$. 
To this end, we investigate the uniform Diophantine approximation and asymptotic Diophantine 
approximation together. We consider the following subset of $[v-\epsilon, v+\epsilon]\cap F(B,y)$
\begin{align*}
F(v, \epsilon, b, B,y):=\{z\in [v-\epsilon, v+\epsilon]:  \|z^n -y_n\| < b^{-n} \text{ for
infinitely many } n \\ 
\text{and } \forall N\gg 1,  \|z^n-y_n\| < B^{-N} \text{ has a solution } 1\leq n \leq N\}. \end{align*}
The proof of Theorem \ref{thm-unif} will be completed by maximizing the lower bounds 
of $F(v, \epsilon, b, B, y)$ with respect to $b>B$. 
\begin{proof}[Proof of Theorem \ref{thm-unif}]
We first construct a subset $F\subset F(v, \epsilon, b, B,y)$.
Suppose that $b=B^{\theta}$ with $\theta>1$. Let  $n_{k}=\lfloor \theta^k \rfloor$.
Consider the points $z$ such that 
\[
\|z^{n_k}-y_{n_k}\|< b^{-n_k}.
\]
Then one can check that $$z\in F(v, \epsilon, b, B,y)=F(v, \epsilon, B^\theta, B, y)
=F(v, \epsilon, b,b^{1 \over \theta}, y).$$

We do the same construction as in Section \ref{sec-asy}. We will obtain a 
Cantor set $F\subset F(v, \epsilon, b,b^{1 \over \theta}, y)$, 
which is the intersection of a nested family of intervals with
\[
m_k(z) = \frac {2 n_{k+1} c_k(z)^{n_{k+1}-1}}{n_k b^{n_k} d_k(z)^{n_k-1}}
\]
and
\[
\varepsilon_k(z) = \left( 1 - \frac 2 {b^{n_{k+1}}}\right) \frac 1 {n_{k+1}d_k(z)^{n_{k+1}-1}},
\]
where $[c_k(z), d_k(z)]$ is the $k$-th level interval containing $z$. 

By the choice of $n_k$, we will have the following estimations:
\[
m_k(z)\geq {2 (\theta^{k+1}-1) c_k(z)^{\theta^{k+1}-1} 
\over \theta^k b^{\theta^k}d_k(z)^{\theta^k-1}} 
\geq \theta  \cdot b^{-\theta^k} \cdot \left({c_k(z) \over d_k(z)}\right)^{\theta^k} 
\cdot c_k(z)^{\theta^{k}(\theta-1)}.
\]
and 
\[
\varepsilon_k(z) \geq {1 \over 2\theta^{k+1}} \cdot d_k(z)^{-\theta^{k+1}}.
\]
Since 
\[
d_k(z)-c_k(z) \leq {b^{-n_k} \over n_k c_k(z)^{n_k-1}} \leq {b^{-\theta^k}}
\]
is much more smaller than $1/\theta^{k}$,
\[
\left({c_k(z) \over d_k(z)}\right)^{\theta^k}
=\left(1-{d_k(z)-c_k(z) \over d_k(z)}\right)^{\theta^k} \geq {1 \over 2}.
\]
Then 
\[
m_k(z)\geq  {\theta \over 2}  \cdot  \left({c_k(z)^{\theta-1} \over b}\right)^{\theta^{k}} 
\geq {\theta \over 2^{\theta+1}} \left({z^{\theta-1} \over b}\right)^{\theta^{k}},
\]
and
\[
m_k(z)\varepsilon_k(z) \geq  {1 \over 4\theta^k}  \cdot  
\left({c_k(z)^{\theta-1} \over b\cdot d_k(z)^\theta}\right)^{\theta^{k}} 
\geq {1 \over 2^{\theta+3}\theta^k} \left({1\over bz}\right)^{\theta^k}.
\]

Thus by Lemma \ref{local}, for any $z\in F$, we have 
\begin{align*}
& \liminf_{k\to\infty} \frac{\log (m_1(z)\cdots m_{k-1}(z))}{-\log m_k(z)\epsilon_k(z)} \\
\geq & \liminf_{k\to\infty} \frac{\big((\theta-1)\log z - \log b\big)
\sum_{j=1}^{k-1} \theta^j}{\theta^k \log bz}\\
= &\liminf_{k\to\infty}{(\theta-1)\log z - \log b \over (\theta-1)\log bz} \cdot {\theta^{k-1}-1 
\over \theta^{k-1}}\\
=& {(\theta-1)\log z - \log b \over (\theta-1)\log bz}.
\end{align*}

Hence, by the relation $b=B^\theta$, we deduce that the Hausdorff dimension of the set 
$F(v, \epsilon, b, B,y)=F(v, \epsilon, B^\theta, B,y)$ is at least equal to 
\[
{(\theta-1)\log (v-\epsilon) - \log b \over (\theta-1)\log (b(v-\epsilon))} 
={\log (v-\epsilon) - {\theta \over \theta-1} \log B \over \log (v-\epsilon) + \theta\log B}.
\]

Taking $\theta \to \infty$ in the left side of the equality, we get the lower bound 
${\log (v-\epsilon)/\log (b(v-\epsilon))}$ for the Hausdorff dimension of 
the set considered in Theorem \ref{main}:
\begin{align*}
&[v-\epsilon, v+\epsilon]\cap E(b, y)\\
=&\{v-\epsilon\leq x\leq v+\epsilon:  \|x^n -y_n\| < b^{-n} \text{ for
infinitely many } n\}.
\end{align*}

By maximizing the right side of the equality 
with respect to $\theta>1$, we obtain the lower bound 
$$
\left({\log (v-\epsilon)-\log B \over \log (v-\epsilon)+\log B}\right)^2
$$ 
for the Hausdorff dimension of the set 
\begin{align*}
&[v-\epsilon, v+\epsilon]\cap F(B,y)\\
=&\{v-\epsilon\leq x\leq v+\epsilon: 
\forall N\gg 1,  ||x^n-y|| < B^{-N} \text{ has a solution } 1\leq n \leq N\}.
\end{align*}
By letting $\epsilon$ tend to $0$, this completes the proof of Theorem \ref{thm-unif}.
\end{proof}

\bigskip
\section{Bad approximation}

In this section, we study the bad approximation properties 
of the sequence $(\{x^n\})_{n\geq 1}$, where $x>1$. 

 Let $q=(q_n)_{n\geq 1}$ be a sequence of positive real numbers 
 and $y=(y_n)_{n\geq 1}$ be an arbitrary sequence of real numbers in $[0,1]$.
 Define
 \[
G(q, y)=\{x>1: \lim_{n\to\infty} \|x^{q_n}-y_n\|=0 \},
\]
and, for $v>1$, define
\[
G(v,q, y)=\{1<x<v: \lim_{n\to\infty} \|x^{q_n}-y_n\|=0 \}.
\]
Recently Baker \cite{Ba} showed that if $q=(q_n)_{n\geq 1}$ is strictly increasing and 
$$\lim_{n\to\infty}(q_{n+1}-q_n) = \infty,$$ 
then the set $G(q, y)$ has Hausdorff dimension $1$. 


We want to generalize Baker's result.
Consider a family of $C^1$ functions $f=(f_n)_{n\geq 1}$ 
from an interval $I \subset \mathbb{R}$ to $\mathbb{R}$ 
such that $f_n'(x)\geq 1$ for all $x\in I$ and for all $n\geq 1$. 
Let $\delta=(\delta_n)_{n\geq 1}$ be a sequence of positive real numbers tending to $0$.  
For $\epsilon>0$, set
\[
G(\epsilon, v,f, y, \delta):=\{v-\epsilon <x<v+\epsilon:  \|f_n(x)-y_n\|\leq \delta_n, \ \forall n\geq 1  \}.
\]
To prove Theorem \ref{bad-general}, we need 
to estimate $\dim_HG(\epsilon, v,f, y, \delta)$.
%
%

\begin{proof}[Sketch proof of Theorem \ref{bad-general}]

{\bf Lower bound}:
We do the same construction as in the proof of the lower bound 
in Theorem \ref{asym-general}. 
If the {right-hand} side inequality in \eqref{Ki} is satisfied, that is, if 
\begin{equation}\label{estimate-righthand}
{\log f_{n+1}'(v) \over \log f_{n}'(v)} \leq{\gamma \over 2\eta(n)},
\end{equation}
for some $\gamma>0$, for large enough $n$, and for $\eta$ defined in (\ref{def:eta}),
then the distortion estimation \eqref{seq-ni} holds and we estimate the dimension in exactly the same way as in Theorem  \ref{asym-general}.

If, however, (\ref{estimate-righthand}) is not satisfied, that is, at some place $f_n'$ is too sparse, 
with $\log f_{n+1}'(v) \gg \log f_n'(v)$ then we can apply 
the idea of Baker (\cite{Ba}, page 69): we add some new functions 
$\tilde{f}_m$ between $f_n$ and $f_{n+1}$, in such a way that the resulting, expanded, 
sequence of their logarithms of derivatives is not too sparse anymore. 
We also add some $\tilde{\delta}_m=1$ for each added $\tilde{f}_m$.
Observe that the {right-hand} side of \eqref{formula:thm-bad} does not change. 
Naturally, {the resulting set $G(\epsilon, v, \tilde{f}, y, \tilde{\delta})$ is exactly the same as $G(\epsilon, v, f, y, \delta)$.
So, for the lower bound, we need only to estimate the lower bound of $\dim_H G(\epsilon, v, \tilde{f}, y, \tilde{\delta})$.}

This means that we can freely assume that (\ref{estimate-righthand}) holds.

We will construct a subset of $G(\epsilon, v,f, y, \delta)$ which is the intersection 
of a nested family of subintervals $I_n(\cdot)$.

For $n=1$, by the property $\|f_1(x)-y_1\|\leq \delta_1$, we take the intervals at level $1$ as 
\[
I_1(k, v, f, y, \delta):=[f_1^{-1}(k+y_1-\delta_1), f_1^{-1}(k+y_1+\delta_1)],
\]
with $k$ being an integer in $[f_1(v-\epsilon)+1, \quad  f_1(v+\epsilon)-1]$. 

Suppose we have constructed the intervals at level $n-1$. Let $[c_{n-1}, d_{n-1}]$ 
be an interval at this level.
A subinterval of $[c_{n-1}, d_{n-1}]$ at level $n$ is 
\[
[f_n^{-1}(k+y_n-\delta_n), f_n^{-1}(k+y_n+\delta_n)],
\]
with $k$ being an integer in $[f_n(c_{n-1})+1, \ f_n(d_{n-1})-1]$.
By continuing this construction, we obtain intervals $I_n(\cdot)$ for all levels.
Finally, the intersection $F$ of these nested intervals is 
obviously a subset of $G(\epsilon, v,f, y, \delta)$.

Let $z\in F$ and $[c_{n}(z), d_{n}(z)]$ be the $n$-th level interval containing $z$. 
Then by (\ref{hyp-1})
\[
m_{n+1}(z) \geq f_{n+1}'(w_{n}) \cdot (d_{n}-c_{n})-2 
\geq f_{n+1}'(w_n) \cdot {2\delta_n \over f_n'(z_n)}-2 \geq 2,
\]
and  
\[
\epsilon_{n+1}(z) \geq  {1-2{ \delta_{n+1}} \over f_{n+1}'(u_n)} \geq {1 \over 2f_{n+1}'(u_n)}.
\]
with $w_n, z_n, u_n\in[c_{n}(z), d_{n}(z)]$.
As we are assuming (\ref{estimate-righthand}), we have \eqref{seq-ni} and then for any $\gamma>0$
\[
m_{n+1}(z) \geq f_{n+1}'(z) \cdot f_{n}'(z)^{-\gamma} \cdot { \delta_n \over f_n'(z_n)}.
\]
%
and
\[
m_{n+1}(z)\epsilon_{n+1}(z) \geq f_{n}'(z)^{-\gamma}\cdot {\delta_n \over 2f_n'(z_n)}
\]
Thus,
\begin{align*}
& \frac{\log (m_2(z)\cdots m_{n}(z))}{-\log m_{n+1}(z)\epsilon_{n+1}(z)} \\
\geq& \frac{\log f_{n}'(z) - \log f_1'(z) + \sum\limits_{j=1}^{n-1}\log \delta_j 
+\sum\limits_{j=1}^{n-1} \log {f_j'(z)^{-\gamma} \over f_j'(z_j)} }{\log 2+\log f_n'(z_n)
+\gamma \log f_n'(z) -\log \delta_n}. 
\end{align*}
By (\ref{hyp-2}), we have
\begin{align*}
& \liminf_{n\to\infty}\frac{\log (m_2(z)\cdots m_{n}(z))}{-\log m_{n+1}(z)\epsilon_{n+1}(z)} \\
\geq& \liminf\limits_{n\to \infty} { \log f_n'(z)+ {\sum\limits_{j=1}^{n-1} \log\delta_j}  
\over \log f_n'(z_n) + \gamma \log f_n'(z)-  {\log \delta_n}}.
\end{align*}
Since $\gamma$ can be chosen arbitrary small and $z_n$ tends to $z$, 
by (\ref{cond-ext}) we have
\begin{align*}
 \liminf_{n\to\infty}\frac{\log (m_2(z)\cdots m_{n}(z))}{-\log m_{n+1}(z)\epsilon_{n+1}(z)} 
\geq& \liminf\limits_{n\to \infty}
{ \log f_n'(z)+ {\sum\limits_{j=1}^{n-1} \log\delta_j}  \over \log f_n'(z) -  {\log \delta_n}}.
\end{align*}
Hence the lower bound of Theorem \ref{bad-general} is obtained by Lemma \ref{local}.

\smallskip

{\bf Upper bound}:  We will apply Lemma \ref{local-upper}. For each basic interval $I_n(z)$, 
by (\ref{cond-ext}), we have for any $\gamma$, for $n$ large enough
\[
{\delta_n \over f_n'(z) f_{n-1}'(z)^\gamma } \leq |I_n(z)| 
\leq {\delta_n f_{n-1}'(z)^\gamma \over f_n'(z)  }.
\]
Thus, 
\[
m_n(z) \leq |I_{n-1}(z)| \cdot f_n'(z) f_{n-1}'(z)^\gamma \leq {\delta_{n-1} f_{n-2}'(z)^\gamma 
\over f_{n-1}'(z)} f_n'(z) f_{n-1}'(z)^\gamma.
\]
Hence,
\[
\prod_{j=2}^nm_j(z) \leq f_n'(z) \cdot \prod_{j=1}^{n-1} \delta_{j} 
\cdot{\prod\limits_{j=1}^{n-1}f_j'(z)^{2\gamma} \over f_{1}'(z)}.
\]
Therefore, by (\ref{hyp-2}),
\[
\liminf_{n\to\infty} \frac{\log (m_1(z)\cdots m_{n}(z))}{-\log |I_n(z)|}\leq  \liminf_{n\to\infty}
{ \log f_n'(z)+ {\sum\limits_{j=1}^{n-1} \log\delta_j}  \over \log f_n'(z) -  {\log \delta_n}}.
\]
By Lemma \ref{local-upper}, we conclude the proof.
\end{proof}

%
%
%
%
%

\medskip

\bibliographystyle{alpha}

\begin{thebibliography}{}

\end{thebibliography}


\begin{thebibliography}{Bug03}

\bibitem{Ba} S. Baker, {\it On the distribution of powers of real numbers modulo $1$}, Uniform Distribution Theory {\bf 10}, no.2,  (2015), 67--75.

\bibitem{BS} J. Barral,  S. Seuret, {\it A localized Jarnik-Besicovich theorem}, Adv. Math. {\bf 226} (4)  (2011), 3191-3215.

\bibitem{BV} V. Beresnevich and S. Velani. {\it A Mass Transference Principle and the Duffin--Schaeffer conjecture for
Hausdorff measures}, Ann. of Math. (2) {\bf 164} (3) (2006), 971--992.


\bibitem{BF} I. Borosh and A. S. Fraenkel. {\it A generalization of Jarn\'ik's theorem on Diophantine approximations}, Indag. Math. 34 (1972), 193?201.


\bibitem{Bu12}
Y. Bugeaud,
Distribution modulo one and Diophantine approximation.
Cambridge Tracts in Mathematics 193, Cambridge, 2012. 


\bibitem{BD} Y. Bugeaud, A. Dubickas, 
{\it On a problem of Mahler and Szekeres on approximation by roots of integers}, 
Michigan Math. J. {\bf 56} (2008), 703--715. 

\bibitem{BuLi} Y. Bugeaud and L. Liao, {\it Uniform Diophantine approximation related to $b$-ary and $\beta$-expansions}, Ergod. Th. \& Dynam. Sys, {\bf 36}, no. 1, (2016), 1--22.

\bibitem{BM} Y. Bugeaud,  N. Moshchevitin, {\it On fractional parts of powers of real numbers close to $1$}. Math. Z. {\bf 271} (2012), no. 3-4, 627--637.

\bibitem{Fa1} K. J. Falconer, Fractal Geometry, Mathematical Foundations
and Application, Wiley, 1990.

\bibitem{Fa2} K. J. Falconer, Techniques in Fractal Geometry, Wiley, 1997.

\bibitem{Kah} J. P. Kahane, {\it Sur la r\'epartition des puissances modulo 1}, C. R. Math. Acad. Sci.
Paris {\bf 352} (2014), no. 5, 383--385.

\bibitem{KL} D. H. Kim and L. Liao, {\it Dirichlet uniformly well-approximated numbers}, preprint, arxiv.org/abs/1508.00520.

\bibitem{Kok} J. F. Koksma, {\it Ein mengentheoretischer Satz \"{u}ber die Gleichverteilung modulo Eins}, Compositio Math. {\bf 2} (1935), 250-258. 

\bibitem{Kok45} J. F. Koksma, {\it Sur la th\'eorie m\'etrique des approximations diophantiques}, Indag. Math., {\bf 7} (1945),
54--70.

\bibitem{MS} K. Mahler and G. Szekeres, {\it On the approximation of real numbers by roots of integers}, Acta Arith.,
{\bf 12} (1967), 315--320.

\bibitem{Po} A.D. Pollington, {\it The Hausdorff dimension of certain sets related to sequences
which are not dense mod $1$}, Quart. J. Math. Oxford Ser. (2) {\bf 31} (1980), 351--361.





%
 


\end{thebibliography}

%
%

\end{document}